\documentclass{amsart}
\usepackage{amsfonts,amssymb,amscd,amsmath,enumerate,verbatim,calc}
\usepackage[all]{xy}

\newcommand{\CM}{Cohen-Macaulay}

\newcommand{\wrt}{with respect to}

\newcommand{\I}{\mathbb{I} }
\newcommand{\n}{\mathfrak{n} }
\newcommand{\m}{\mathfrak{m} }

\newcommand{\q}{\mathfrak{q} }
\newcommand{\R}{\mathcal{R} }
\newcommand{\C}{\mathcal{C} }
\newcommand{\Z}{\mathbb{Z} }
\newcommand{\X}{\mathcal{X} }
\newcommand{\rt}{\rightarrow}

\newcommand{\ov}{\overline}

\newcommand{\wh}{\widehat }

\newcommand{\Om}{\Omega}
\newcommand{\image}{\operatorname{image}}

\newcommand{\projdim}{\operatorname{projdim}}

\newcommand{\depth}{\operatorname{depth}}
\newcommand{\Tor}{\operatorname{Tor}}

\newcommand{\ann}{\operatorname{ann}}

\newcommand{\Supp}{\operatorname{\underline{Supp}}}
\newcommand{\Spec}{\operatorname{Spec}}
\newcommand{\CMS}{\operatorname{\underline{CM}}}

\newcommand{\Hom}{\operatorname{Hom}}
\newcommand{\sHom}{\operatorname{\underline{Hom}}}
\newcommand{\Ext}{\operatorname{Ext}}

\theoremstyle{plain}

\newtheorem{theorem}{Theorem}[section]
\newtheorem{corollary}[theorem]{Corollary}
\newtheorem{lemma}[theorem]{Lemma}
\newtheorem{proposition}[theorem]{Proposition}

\theoremstyle{definition}

\newtheorem{s}[theorem]{}

\newtheorem{example}[theorem]{Example}

\theoremstyle{remark}

\begin{document}

\title[Hilbert polynomials]{Derived functors and Hilbert polynomials over hypersurface rings-II }
\author{Tony~J.~Puthenpurakal}
\date{\today}
\address{Department of Mathematics, IIT Bombay, Powai, Mumbai 400 076}

\email{tputhen@math.iitb.ac.in}

 \begin{abstract}
Let $(A,\m)$ be a hypersurface local ring of dimension $d \geq 1$, $N$ a perfect $A$-module and let $I$ be an  ideal in $A$ with $\ell(N/IN)$ finite. We show that there is a integer $r_I \geq -1$ (depending only on $I$ and $N$) such that if $M$ is any non-free maximal \CM \ (=  MCM)
$A$-module the  functions $n \rt \ell(\Tor^A_1(M, N/I^{n+1}N))$, $n \rt \ell(\Ext^1_A(M, N/I^{n+1}N))$ and $n \rt \ell(\Ext^{d+1}(N/I^{n+1}N, M))$ (which are all of polynomial type) has degree $r_I$.
Surprisingly a  key ingredient is the classification of thick subcategories of the  stable category of MCM $A$-modules (obtained by Takahashi, see \cite[6.6]{T}).
\end{abstract}
 \maketitle
\section{introduction}
Dear Reader, while reading this paper it is a good idea to have part 1 of this paper, see \cite{P5}.

In this paper we prove  few surprising  results.  Recall $A$ is said to be a hypersurface ring if its completion $\wh{A} = Q/(f)$ where $(Q, \n)$ is a regular local ring and $f \in \n^2$ is non-zero. We set the degree of the zero polynomial to be $-1$. In part one of this paper we proved
\begin{theorem}\label{p1}(see \cite[1.1, 1.2, 1.4]{P5}). Let  $(A,\m)$ be a hypersurface local ring of dimension $d \geq 1$ and let $I$ be an $\m$-primary ideal. Then there is an  integer $r_I \geq -1$ (depending only on $I$) such that if $M$ is any non-free  MCM
$A$-module then the functions $n \rt \ell(\Tor^A_1(M, A/I^{n+1}))$, $n \rt \ell(\Ext^1_A(M, A/I^{n+1}))$ are of degree
$r_I$.
\end{theorem}
We also proved (see \cite[1.3]{P5}) that  there exists $e_I \geq -1$ (depending only on $I$)  such that if $M$ is any MCM $A$-module free on $\Spec^0(A) = \Spec(A) \setminus \{\m \}$, the function $n \rt \ell(\Ext^{d+1}(A/I^{n+1}, M))$
is of degree $e_I$.

\s The motivation of this paper is to extend the results in part one of this paper.
The first question we address is whether  if $M$ is any MCM $A$-module then is  the function $n \rt \ell(\Ext^{d+1}(A/I^{n+1}, M))$
of degree $e_I$? Furthermore is $r_I = e_I$? In this paper we give affirmative answers to both of these questions.

Next we look at the polynomial type functions  $n \rt \ell(\Tor^A_1(M, N/I^{n+1}N))$, $n \rt \ell(\Ext^1_A(M, N/I^{n+1}N))$ and $n \rt \ell(\Ext^{d+1}(N/I^{n+1}N, M))$ (where either $M$ is free on $\Spec^0(A)$ or if $\ell(N/IN)$ is finite) and ask analogous questions.

\s Let $M$ be a MCM $A$-module, $N$ a finitely generated $A$-module and let $I$ be an ideal in $A$.
Set for $n \geq 0$
\begin{align*}
  t_{I, N}(M,n) &=  \ell(\Tor^A_1(M, N/I^{n+1}N)) \\
  s_{I, N}(M,n) &= \ell(\Ext^1_A(M, N/I^{n+1}N)) \\
  e_{I, N}(M, n) &= \ell(\Ext^{d+1}_A( N/I^{n+1}N, M)).
\end{align*}
\emph{whenever} the corresponding lengths are finite. There  are two general cases when these lengths are finite.

1. The lengths are finite if $M$ is free on $\Spec^0(A)$.

2. The lengths are finite if $\ell(N/IN)$ is finite.

\s We first consider the case when $M$ is free on $\Spec^0(A)$. Let $\CMS(A)$ denote the stable category of MCM $A$-modules and let $\CMS^0(A)$ be the thick subcategory of MCM $A$-modules which are free on $\Spec^0(A)$. Let $\Omega^1(M)$ denote the first syzygy of $M$. We prove
\begin{theorem}\label{m1}
Let $(A,\m)$ be a hypersurface ring of dimension $d$. Let $N$ be a finitely generated $A$-module and let $I$ be an ideal in $A$.
 There exists integers $r_{I, N}^0$ and $s_{I, N}^0$ (both $\geq -1$) depending only on $I$ and $N$ such that if $M \in \CMS^0(A)$ is  non-free then
$\deg t_{I,N}(M \oplus \Omega^1(M), -) = r_{I, N}^0$ and $\deg s_{I,N}(M \oplus \Omega^1(M), -) = s_{I, N}^0$.
\end{theorem}
We give an example which shows that $\deg t_{I,N}(M, -) $ need not be constant on $\CMS^0(A)$, see \ref{lucho}. Before analyzing the case of $e_{I, N}(M, -)$ we first need to show that it is of polynomial type.
We prove:
\begin{lemma}
\label{e-growth} Let $(A,\m)$ be a Gorenstein local ring of dimension $d$ and let $M$ be a MCM $A$-module free on $\Spec^0(A)$. Let $N$ be a finitely generated $A$-module and let $I$ be an ideal in $A$. Fix $j \geq d +1$. Then the function $n \mapsto \ell(\Ext^{j}_A( N/I^{n+1}N, M))$ is of polynomial type and of degree $\leq d$.
\end{lemma}
Then we show
\begin{theorem}\label{m2}
Let $(A,\m)$ be a hypersurface ring of dimension $d$. Let $N$ be a finitely generated $A$-module and let $I$ be an ideal in $A$.
 There exists integer $e_{I, N}^0$ ($\geq -1$) depending only on $I$ and $N$ such that if $M \in \CMS^0(A)$ is a non-free then
$\deg e_{I,N}(M \oplus \Omega^1(M), -) = e_{I, N}^0$.
\end{theorem}

\s We now consider the case when $\ell(N/IN)$ is finite. In this case we can prove that $t_{I, N}(M, n) = t_{I, N}(\Omega^1(M), n)$ for all $n \geq 0$. Similarly $s_{I, N}(M, n) = s_{I, N}(\Omega^1(M), n)$ and
$e_{I, N}(M, n) = e_{I, N}(\Omega^1(M), n)$ for all $n \geq 0$; see \ref{basic}(2). We may ask if $\deg t_{I, N}(M, -) = r_{I, N}^0$ for all non-free MCM $A$-modules. In this generality the result does not hold. We show:
\begin{example}\label{ex}(see Section 5).
Let $(A, \m)$ be a hypersurface ring of dimension $d \geq 1$ and \emph{NOT} an isolated singularity. Let $P$ be a prime ideal in $A$ with $t = \dim A/P > 0$ and $A_P$ NOT regular local.
Set $N = A/P$, $I = \m$ and $M = \Omega^{d}(A/P) \oplus \Omega^{d+1}(A/P)$. Then $\deg t_{\m, N}(M, -) = t$ while we can show $r_{\m, N}^0 \leq t -1$.
\end{example}
Next we consider the problem of finding  classes of modules $N$ where \\ $\deg t_{I, N}(M, n) = r_{I, N}^0$ for all non-free MCM $A$-modules.  We solve the problem in two steps.
Let
$$\X(N)  = \{ M \mid M \in \CMS(A) \ \text{and} \ \ell(\Tor^A_i(M, N)) < \infty \ \text{for all}\ i \geq 1 \}.$$
It is easy to see that $\X(N)$ is a thick subcategory of $\CMS(A)$ containing $\CMS^0(A)$. First we show
\begin{theorem}\label{xi}
(with hypotheses as in \ref{m1}). Assume $\ell(N/IN)$ is finite. Let $M \in \X(N)$ be non-free. Then $\deg t_{I, N}(M,-) = r_{I, N}^0$ and $\deg s_{I, N}(M,-) = s_{I, N}^0$
\end{theorem}
So our problem reduces to find conditions on $N$ which ensure $\X(N) = \CMS(A)$. In this regard we show:
\begin{proposition}\label{XNmax}
Let $(A, \m)$ be a hypersurface ring and let $N$ be a finitely generated $A$-module. The following assertions are equivalent:
\begin{enumerate}[\rm (1)]
  \item $\X(N) = \CMS(A)$.
  \item $\projdim_{A_P} N_P$ is finite for all primes $P \neq \m$.
\end{enumerate}
\end{proposition}
As a consequence we obtain
\begin{corollary}
  \label{m1-gen}
Let $(A,\m)$ be a hypersurface ring of dimension $d$. Let $N$ be a finitely generated $A$-module and let $I$ be an ideal in $A$ with $\ell(N/IN)$ finite. Assume $\projdim_{A_P} N_P$ is finite for all primes $P \neq \m$. Then
 there exists integers $r_{I, N}^0$ and $s_{I, N}^0$ (both $\geq -1$) depending only on $I$ and $N$ such that if $M \in \CMS(A)$ is  non-free then
$\deg t_{I,N}(M, -) = r_{I, N}^0$ and $\deg s_{I,N}(M, -) = s_{I, N}^0$.
\end{corollary}

The function $ e_{I,N}(M, -) $ behaves well under some more restrictive hypotheses on $N$.
We show
\begin{theorem}\label{m2-gen}
Let $(A,\m)$ be a hypersurface ring of dimension $d$. Let $N$ be a perfect $A$-module and let $I$ be an ideal in $A$ with $\ell(N/IN)$ finite. Then
 there exists integer $e_{I, N}^0$  ($\geq -1$) depending only on $I$ and $N$ such that if $M \in \CMS(A)$ is  non-free then
$\deg e_{I,N}(M, -) = e_{I, N}^0$.
\end{theorem}

In part one of this paper we had proved $ r_{I, A}^0 = s_{I, A}^0$ when $I$ is an $\m$-primary ideal in $A$ and $N = A$.
in this paper we extend and generalize this result. We prove that
\begin{theorem}\label{m3-gen}
Let $(A,\m)$ be a hypersurface ring of dimension $d$. Let $N$ be a perfect $A$-module and let $I$ be an ideal in $A$ with $\ell(N/IN)$ finite.  Then
 $$ r_{I, N}^0 = s_{I, N}^0 = e_{I, N}^0. $$
\end{theorem}

Here is an overview of the contents of this paper. In section two we discuss a few preliminaries that we need.
In section three we first recall some results of Theodorescu. Then we prove Lemma \ref{e-growth}. In the next section we prove Theorems \ref{m1}, \ref{m2}. In section five we prove Example \ref{ex}. In the next section we give some preliminary results of the category $\X(N)$. In section seven we prove Theorem \ref{xi}. This implies \ref{m1-gen}. In the next section we prove Theorem \ref{m2-gen}. Finally in section nine we prove Theorem \ref{m3-gen}.
\section{Preliminaries}
In this section we discuss a few preliminary results that we need.
We use \cite{N} for notation on triangulated categories. However we will assume that if $\mathcal{C}$ is a triangulated category then $\Hom_\mathcal{C}(X, Y)$ is a set for any objects $X, Y$ of $\mathcal{C}$.

\s \label{t-f} Let $\C$ be a skeletally  small triangulated category  with shift operator $\Sigma$ and let $\I(\C)$ be the set of isomorphism classes of objects in $\C$. By a \emph{weak triangle function} on $\C$ we mean a function $\xi \colon \I(\C) \rt \Z$ such that
\begin{enumerate}
  \item $\xi(X) \geq 0$ for all $X \in \C$.
  \item $\xi(0) = 0$.
  \item $\xi(X \oplus Y) = \xi(X) + \xi(Y)$ for all $X, Y \in \C$.
  \item $\xi(\Sigma X ) = \xi(X)$ for all $X \in \C$.
  \item If $X \rt Y \rt Z \rt \Sigma X $ is a triangle in $\C$ then
   $\xi(Z) \leq \xi(X) + \xi(Y)$.
\end{enumerate}
\s Set $$\ker \xi = \{ X \mid \xi(X) = 0 \}.$$
The following result (essentially an observation) is a crucial ingredient in our proof of Theorems \ref{m1}, \ref{m2}.
\begin{lemma}(see \cite[2.3]{P4} )
\label{ker-lemma}(with hypotheses as above)
$\ker \xi $ is a thick subcategory of $\C$.
\end{lemma}

\s Let $(A,\m)$ be a hypersurface ring. Let $N$ be a finitely generated $A$-module and let $I$ be an  ideal in $A$.
Let $M$ be a MCM $A$-module. Assume either $M$ is free on $\Spec^0(A)$ or $\ell(N/IN)$ is finite.
Set for $n \geq 0$
\begin{align*}
  t_{I, N}(M,n) &=  \ell(\Tor^A_1(M, N/I^{n+1}N)) \\
  s_{I, N}(M,n) &= \ell(\Ext^1_A(M, N/I^{n+1}N)) \\
  e_{I, N}(M, n) &= \ell(\Ext^{d+1}_A( N/I^{n+1}N, M)).
\end{align*}
Let $\Omega^i_A(M)$ denote the $i^{th}$-syzygy of $M$.
We prove
\begin{lemma}\label{basic}
(with hypotheses as above)
\begin{enumerate}[\rm (1)]
  \item For all $n \geq 0$  the functions $t_{I, N}(-,n), s_{I, N}(-,n)$ and $e_{I, N}(-,n)$ are functions on $\CMS(A)$
  \item Assume $\ell(N/IN)$ is finite.
  For all $n \geq 0$ we have $t_{I, N}(M,n) = t_{I, N}(\Omega^1_A(M),n)$, \\ $s_{I, N}(M,n) = s_{I, N}(\Omega^1_A(M),n)$ and $e_{I, N}(M,n) = e_{I, N}(\Omega^1_A(M),n)$.
\end{enumerate}
\end{lemma}
\begin{proof}
(1) Let $E = M \oplus F = L \oplus G$ where $F, G$ are free $A$-modules. Then by definition
$t_{I, N}(E, n) = t_{I, N}(M,n) = t_{I, N}(L,n)$. Thus $t_{I, N}(-, n)$ is a function on $\CMS(A)$.

 The proof for assertions on $s_{I, N}(-, n)$ and $e_{I, N}(-, n)$ are similar.

(2) We may assume that $M$ has no free summands. Set $L = \Omega_A^1(M)$.  Let $ 0 \rt L  \rt F \rt M \rt 0$ be the minimal presentation of $M$ with $F = A^r$. Then note as $A$ is a hypersurface ring and $M$ is MCM without free summands we get that a minimal presentation of $L$ is as follows $0 \rt M \rt G \rt L \rt 0$ where $G = A^r$,  see \cite[6.1(ii)]{E}.
By using the first exact sequence we get
\begin{align*}
0 \rt \Tor^A_1(M, N/I^{n+1}N) &\rt L\otimes N/I^{n+1}(L\otimes N)\rt (N/I^{n+1}N)^r\\
&\rt M\otimes N/I^{n+1}(M\otimes N) \rt 0.
\end{align*}
So we have
\[
t_{I, N}(M,n) = \ell(L\otimes N/I^{n+1}(L\otimes N)) + \ell(M\otimes N/I^{n+1}(M\otimes N)) - r\ell(N/I^{n+1}N).
\]
Using the second exact sequence we find that $t_{I, N}(M,n) = t_{I, N}(L,n)$. The result follows.

The proof for assertions on $s_{I, N}(-, n)$ and $e_{I, N}(-, n)$ are similar.
\end{proof}

\section{Polynomial growth}
In this section we first recall some results of Theodorescu. Finally we prove Lemma \ref{e-growth}. Let $F_N(I) = \bigoplus_{n \geq 0}I^nN/\m I^nN$  the fiber-cone of $I$ \wrt \ $N$ and let  $s_N(I) = \dim F_N(I)$. We note that $s_N(I) \leq \dim N$.
The first result is
\begin{theorem}\label{complex}(\cite[Proposition 3]{Theo}):
Let $(A, \m)$ be  a Noetherian local ring. Let $\mathcal{C} \colon F_2 \rt F_1 \rt F_0$ be a complex with $F_1, F_0$ finitely generated $A$-modules. Let
$N$ be another finitely generated $A$-module and let $I$ be an ideal of $A$.
Assume $H_1(\mathcal{C} \otimes N/I^nN)$ has finite length for all $n \geq 1$. Then
\begin{enumerate}[\rm (1)]
  \item The function $n \rt \ell(H_1(\mathcal{C} \otimes N/I^n N))$ of polynomial type, say of degree $u_N$.
  \item We have $u_N \leq \max \{ \dim H_1(\mathcal{C} \otimes N), s_N(I) -1 \}$.
  \item If $\dim H_1(\mathcal{C} \otimes N)\geq s_N(I) $ then the inequality in $\rm{(2)}$ becomes an equality.
\end{enumerate}
\end{theorem}

As a corollary Theodorescu obtains:
\begin{corollary}
\label{tor-ext-deg}(\cite[Corollary 4]{Theo}:) Let $(A, \m)$ be  a Noetherian local ring. Let $M, N$ be finitely generated $A$-modules and let $I$ be an ideal in $A$. Fix $i \geq 0$.
\begin{enumerate}[\rm (I)]
  \item Assume $\ell(\Tor^A_i(M, N/I^n N))$ is finite for all $n \geq 1$. Then
  \begin{enumerate}[\rm (i)]
    \item The function $n \rt \ell(\Tor^A_i(M, N/I^n N))$ is of polynomial type say of degree $t_I(M, N)$.
    \item $t_I(M, N) \leq \max \{ \dim \Tor^A_i(M, N), s_N(I) - 1\}$.
    \item If $\dim \Tor^A_i(M, N) \geq s_N(I)$ then the inequality in $\rm{(ii)}$ becomes an equality.
  \end{enumerate}
  \item Assume $\ell(\Ext^i_A(M, N/I^n N))$ is finite for all $n \geq 1$. Then
  \begin{enumerate}[\rm (i)]
    \item The function $n \rt \ell(\Ext^i_A(M, N/I^n N))$ is of polynomial type say of degree $e_I(M, N)$.
    \item $e_I(M, N) \leq \max \{ \dim \Ext^i_A(M, N), s_N(I) - 1\}$.
    \item If $\dim \Ext^i_A(M, N) \geq s_N(I)$ then the inequality in $\rm{(ii)}$ becomes an equality.
  \end{enumerate}
\end{enumerate}
\end{corollary}

We also need a Lemma due to Theodorescu.
\begin{lemma}\label{vee}(\cite[Lemma 1]{Theo}:)
Let $(A, \m)$ be a Noetherian local ring and let $I$ be an ideal in $A$. Let $\mathcal{C}$ be a complex of (not-necessarily finitely generated) $A$-modules. Let $E$ be an injective $A$-module.
and write $M^\vee = \Hom_A(M, E)$ for any $A$-module $M$. Then we have an isomorphism of complexes
$$ \Hom(A/I, \mathcal{C})^\vee \cong \mathcal{C}^\vee \otimes A/I.$$
\end{lemma}
Next we prove
Lemma \ref{e-growth}. We restate it here for the convenience of the reader.
\begin{lemma}
\label{e-growth-body} Let $(A,\m)$ be a Gorenstein local ring of dimension $d$ and let $M$ be a MCM $A$-module free on $\Spec^0(A)$. Let $N$ be a finitely generated $A$-module and let $I$ be an ideal in $A$. Fix $j \geq d +1$. Then the function $n \mapsto \ell(\Ext^{j}_A( N/I^{n+1}N, M))$ is of polynomial type and of degree $\leq d$.
\end{lemma}
\begin{proof}
We have nothing to show if $I = A$ as the relevant function is zero. So assume that $I \subseteq \m$. We first consider the case when $d = 0$. Then $I^n N = 0$ for all $n \gg 0$.
Thus $\ell(\Ext^{j}_A( N/I^{n+1}N, M)) = \ell(\Ext^j_A(N, M))$ for all $n \gg 0$. The result follows.

So we assume $d \geq 1$. Let $\mathcal{C}$ be the minimal injective resolution of $M$. Then note that for $i \geq d$ we have $\mathcal{C}^i = E^{r_i}$ where $E$ is the injective hull of $k$ and $r_i$ is a finite non-negative integer. Let $(-)^\vee = \Hom(-, E)$.
Let $\mathcal{D} = \Gamma_\m(\mathcal{C})$ be the complex obtained by applying the $\m$-torsion  functor $\Gamma_\m(-)$ on $\mathcal{C}$.
As $j \geq d +1$ we have
\begin{align*}
  \ell(\Ext^j(N/I^{n+1}N, M) &= \ell(H^j(\Hom(N/I^{n+1} N, \mathcal{D}))) \\
  &=  \ell(H^j(\Hom(A/I^{n+1} , \Hom(N,\mathcal{D})))) \\
   &= \ell(H_{-j}(\Hom(A/I^{n+1} , \Hom(N,\mathcal{D})))^\vee) \\
  &= \ell(H_{-j}(\Hom(N, \mathcal{D})^\vee \otimes_A A/I^{n+1})).
\end{align*}
The last equality is due to Lemma \ref{vee}. We note that $\Hom(N, \mathcal{D})^\vee $ is a complex of finitely generated modules over $\wh{A}$ the completion of $A$. Furthermore we note that if $U$ is an $\wh{A}$-module then $U\otimes_A A/I = U\otimes_{\wh{A}}\wh{A}/I\wh{A}$. The result now follows from \ref{complex}.
\end{proof}
The above result does not hold when $j = d$.
\begin{example}
 We take $A$ to be a regular local ring and $M = N = A$. Then $\ell(\Ext^d_A(A/I^n, A)) = H^0_\m(A/I^n)$. This function need not be of polynomial type, see \cite[3.2]{CT}.
\end{example}

\section{Proof of Theorems \ref{m1}, \ref{m2}}
In this section we give
\begin{proof}[Proofs of Theorems \ref{m1}, \ref{m2}]
We first note that for any MCM $M$ we have
$$\deg t_{ I, N}(M \oplus \Omega^1(M), -) \leq \dim N - 1, $$ see
\ref{tor-ext-deg}.
 Recall the degree of the zero polynomial is $-1$. Set
 $$ r = \max \{\deg  t_{I, N}(M\oplus \Omega^1(M), -) \mid M \in  \CMS^0(A) \}. $$
If $r = -1$ then we have nothing to show. So assume $r \geq 0$.  For $M \in \CMS^0(A)$ define
\[
\xi(M) = \lim_{n \rt \infty} \frac{r!}{n^r} t_{ I, N}(M \oplus \Omega^1(M), n)
\]
Then note that $\xi(M)$ is finite non-negative integer. We also note that $\xi(M) = 0$ if and only if $\deg  t_{I, N}(M\oplus \Omega^1(M), -) \leq r -1$.

Claim: $\xi(-)$ is a weak triangle function on $\CMS^0(A)$; see \ref{t-f}.\\
Assume the claim for the time being. We note that by construction there exists $L \in \CMS^0(A)$ with $r = \deg  t_{I, N}(L\oplus \Omega^1(L), -)$.
So $\ker \xi \neq \CMS^0(A)$. Also $\ker \xi$ is a thick subcategory of $\CMS^0(A)$.  As $\CMS^0(A)$ has no proper thick subcategories, see \cite[6.6]{T}, it follows that $\ker \xi = 0$.
Therefore
$\deg t_{ I, N}(M \oplus \Omega^1(M), -) = r$ for all non-free MCM $M \in \CMS^0(A)$.

It remains to show $\xi(-)$ is a weak triangle function on $\CMS^0(A)$. The first three properties is trivial to show.
We have $\xi(M) = \xi(\Omega^{-1}(M))$ as $M$ has period two.
Let $M \rt U \rt V \rt \Omega^{-1}(M)$ be a triangle in $\CMS^0(A)$. Notice we have an exact sequence of $A$-modules
$$0 \rt U \rt V\oplus F \rt \Omega^{-1}(M) \rt 0, \quad \text{where $F$ is free}.$$
So we have an exact sequence
 $$0 \rt \Omega(U) \rt \Omega(V)\oplus G\rt M \rt 0, \quad \text{where $G$ is free}.$$
 Taking direct sum of the above two exact sequences we obtain an exact sequence
 $$ 0 \rt U \oplus\Omega(U) \rt V \oplus\Omega(V)\oplus F \oplus G \rt \Omega^{-1}(M)\oplus M \rt 0.$$
 Note $\Omega^{-1}(M) = \Omega(M)$. It follows that
 $$t_{I, N}(V \oplus \Omega(V), n) \leq t_{I, N}(U \oplus \Omega(U), n)  + t_{I, N}(M \oplus \Omega(M), n).$$
 The result follows.

 The proof for the functions $s_{I, N}(M \oplus \Omega(M), -)$ and $e_{I, B}(M \oplus \Omega(M), -)$ is the same. In the first case degree is bounded above by $\dim N - 1$, see \ref{tor-ext-deg}
 and in the second case the degree is bounded above by $\dim A$, see \ref{e-growth}. Rest all the arguments are similar.
\end{proof}

\s \label{lucho} We show that $t_{I, N}(M , -)$ need not be constant. Let $Q = k[[u_1, x_1]]$. Set $A = Q/(u_1x_1)$. Set $M = N = A/(u_1)$. Also set $I = 0$. We have for $i \geq 1$;
$\Tor^A_{2i}(M, M) = 0$ and $\Tor^A_{2i-1}(M, M) = k$; see \cite[4.3]{AB}.
\section{example \ref{ex}}
We recall the example here.
Let $(A, \m)$ be a hypersurface ring of dimension $d \geq 1$ and \emph{NOT} an isolated singularity. Let $P$ be a prime ideal in $A$ with $t = \dim A/P > 0$ and $A_P$ NOT regular local.
Set $N = A/P$, $I = \m$ and $M = \Omega^{d}(A/P) \oplus \Omega^{d+1}(A/P)$.

We note $s_N(\m) = \dim N$. Let $E \in \CMS^0(A)$. Then by Corollary \ref{tor-ext-deg} we get $r_{\m, N}(E, -) \leq \dim N - 1 = t -1$. So $r_{\m, N}^0 \leq t -1$.

We note that $\Tor^A_1(M, N)_P \neq 0$ by  \cite[1.9]{HW} (also see \cite[1.1]{M}). So \\ $\dim \Tor^A_1(M, N) = t = s_N(\m)$.
Then $\deg t_{\m, N}(M, -) = t$ by  Corollary \ref{tor-ext-deg}.

\section{An investigation into the category $\X(N)$}
In this section we give some preliminary results of the category $\X(N)$. We first show:
\begin{proposition}\label{pdimf}
Let $(A, \m)$ be a Gorenstein local ring of dimension $d$  and let $M$ be a MCM $A$-module. If $N$ is an $A$-module with $\projdim_A N < \infty$ then \\ $\Tor^A_i(M, N) = 0$ for all $ i \geq 1$ and $\Ext_A^i(M, N) = 0$ for all $ i \geq 1$.
\end{proposition}
\begin{proof}
We have
$$ \Tor^A_i(M, N) = \Tor^A_{i + d + 1}(\Omega^{-d -1}M, N) = 0.$$
The assertion for Ext is similar.
\end{proof}
Next we show:
\begin{proposition}\label{ext-x}
Let $(A, \m)$ be a hypersurface local ring of dimension $d$  and let $N$ be an $A$-module.
If $M \in \X(N)$ then $\ell(\Ext^i_A(M, N)) < \infty $ for all $i \geq 1$.
\end{proposition}
\begin{proof}
Let $P$ be a prime ideal of $A$ with $P \neq \m$. As $M \in \X(N)$ we have
$$\Tor^{A_P}_i(M_P, N_P) = \Tor^A_i(M, N)_P = 0,  \quad \text{for all} \ i \geq 1. $$
By \cite[1.9]{HW} (also see \cite[1.1]{M})  it follows that $M_P$ is free or $\projdim_{A_P} N_P < \infty$ or one among $M_P$ or $N_P$ is zero. By \ref{pdimf} it follows that $\Ext_{A_P}^i(M_P, N_P) = 0$ for $i \geq 1$.
The result follows.
\end{proof}
Next we give a proof of Proposition \ref{XNmax}. We restate it for the convenience of the reader.
\begin{proposition}\label{XNmax-body}
Let $(A, \m)$ be a hypersurface ring of dimension $d$ and let $N$ be a finitely generated $A$-module. The following assertions are equivalent:
\begin{enumerate}[\rm (1)]
  \item $\X(N) = \CMS(A)$.
  \item $\projdim_{A_P} N_P$ is finite for all primes $P \neq \m$.
\end{enumerate}
\end{proposition}
\begin{proof}
The assertion (2) $\implies $ (1) follows from \ref{pdimf}.

We now show (1) $\implies $ (2). Suppose if possible $\projdim N_P = \infty $ for some prime $P \neq \m$. Set $M = \Omega^{d}(N)$. By our hypothesis $M \in \X(N)$. In particular
$$\Tor^{A_P}_i(M_P, N_P) = \Tor^A_i(M, N)_P = 0,  \quad \text{for all} \ i \geq 1. $$
By \cite[1.9]{HW} (also see \cite[1.1]{M}) it follows that $M_P$ is free or $\projdim_{A_P} N_P < \infty$ or one among $M_P$ or $N_P$ is zero. But $M_P, N_P$ are non-zero and of infinite projective dimension. This is a contradiction. The result follows.
\end{proof}

\s\label{ex-X} \emph{Examples:} We give a few examples such that $\X(N) = \CMS(A)$, equivalently  $\projdim_{A_P} N_P$ is finite for all primes $P \neq \m$.
\begin{enumerate}
\item If $\projdim N$ is finite then trivially $\X(N) = \CMS(A)$.
  \item  Let $N$ be a syzygy of a finite length module. Then $N_P$ is free for all $P \neq \m$. So  $\X(N) = \CMS(A)$.
  \item If $N \in \CMS^0(A)$ then by definition $N_P$ is free for all $P \neq \m$. So  $\X(N) = \CMS(A)$.
\end{enumerate}
\section{Proof of Theorem \ref{xi}}
In this section we give a proof of Theorem \ref{xi}. We will first need the following result.
\begin{lemma}\label{ann}
Let $(A,\m)$ be a hypersurface, $N$ a finitely generated $A$-module and $I$ an ideal in $A$ with $\ell(N/IN)$ finite. Let $M \in \X(N)$. Then there exists $a \geq 1$ with
$\m^a\Tor^A_i(M, N/I^{n}N) = 0$ for all $i \geq 1$ and all $n \geq 1$. Similarly there exists $b \geq 1$ with
$\m^b\Ext^i_A(M, N/I^{n}N) = 0$ for all $i \geq 1$ and all $n \geq 1$.
\end{lemma}
\begin{proof}
As $M$ has period two it suffices to prove the result for $i = 1, 2$.

 Set $L_1(M) = \bigoplus_{n \geq 1}\Tor^A_1(M, N/I^{n}N)$. Let $\R(I) = A[It]$ be the Rees algebra of $I$.
Let $\R(N) = \bigoplus_{n \geq 0}I^nN$ be the Rees module of $N$. We have an exact sequence of $\R(I)$-modules
\begin{equation*}
 0 \rt \R(N) \rt N[t] \rt \bigoplus_{n \geq 1}N/I^nN \rt 0. \tag{*}
\end{equation*}
So we obtain a sequence of $\R(I)$-modules
\[
\Tor^A_1(M, N)[t] \rt L_1(M) \xrightarrow{f} \R(N)\otimes M.
\]
We note that $\R(N)\otimes M$ is a finitely generated $\R(I)$-module. So $D = \image(f)$ is also a finitely generated $\R(I)$-module. As $\ell(D_n)$ is finite for all $n \geq 1$ it follows that
there exists $r$ such that $\m^rD_n = 0$ for all $n \geq 1$. Also as $\Tor^A_1(M, N)$ has finite length we have $\m^s\Tor^A_1(M, N) = 0$. Take $l = r + s$. Then $\m^lL_1(M)_n = 0$ for all $n \geq 1$.

We note that $L_2(M) = \bigoplus_{n \geq 1}\Tor^A_2(M, N/I^{n}N) = L_1(\Omega(M))$. So an argument similar to above yields $t$ with $\m^tL_2(M)_n = 0$ for all $n \geq 1$. The result follows from taking $a$ to be maximum of $l$ and $t$.

Set $E^1(M) = \bigoplus_{n \geq 1}\Ext_A^1(M, N/I^{n}N)$. Applying $\Hom_A(M, -)$ to the short exact sequence $(*)$ we obtain an exact sequence of $\R(I)$-modules
\[
\Ext^1_A(M, N)[t] \rt E_1(M) \xrightarrow{g} \Ext^2_A(M, \R(N)).
\]
We note that  $\Ext^2_A(M, \R(N))$ is a finitely generated $\R(I)$-module. So $W = \image(g)$ is also a finitely generated $\R(I)$-module. As $\ell(W_n)$ is finite for all $n \geq 1$ it follows that
there exists $u$ such that $\m^uW_n = 0$ for all $n \geq 1$. Also as $\Ext^1_A(M, N)$ has finite length, see \ref{ext-x}, we have $\m^v\Ext_A^1(M, N) = 0$. Take $w = u + v$. Then $\m^wE^1(M)_n = 0$ for all $n \geq 1$.

We note that $E^2(M) = \bigoplus_{n \geq 1}\Ext^2_A(M, N/I^{n}N) = E^1(\Omega(M))$. So an argument similar to above yields $c$ with $\m^cE^2(M)_n = 0$ for all $n \geq 1$. The result follows from taking $b$ to be maximum of $w$ and $c$.
\end{proof}
\s We also need the following notion. Let $M \in \CMS(A)$. Let
\[
\Supp(M) = \{ P \mid  M_P \ \text{is not free} \ A_P-\text{module} \}.
\]
If $I$ is an ideal in $A$ then set $V(I) = \{ P \in \Spec(A) \mid P \supseteq I \}$.
It is readily verified that $\Supp(M) = V(\ann \sHom(M, M))$.

Next we give
\begin{proof}[Proof of Theorem \ref{xi}]
If $D$ is an $A$-module,
for $i \geq 1$, set \\ $L_i(D) = \bigoplus_{n \geq 1}\Tor^A_i(D, N/I^{n}N)$ a $\R(I)$-module.

We first show that  $\deg t_{I, N}(M,-) = r_{I, N}^0$ if $M \in \X(N)$ is non-free.
We prove this assertion by induction on $\dim \Supp(M)$. If $\dim \Supp(M) = 0$ then $M$ is free on $\Spec^0(A)$. In this case we have nothing to show.

Now assume $\dim \Supp(M) > 0$.
By Lemma \ref{ann} there exists $l$ such that \\ $\m^l L_i(M) = 0$ for all $i \geq 1$.

Let
$$x \in \m^l \setminus \bigcup_{ \stackrel{P \supseteq \ann \sHom(M, M)}{P \ \text{minimal}}} P.$$
Let $M \xrightarrow{x} M \rt U \rt \Om^{-1}(M)$ be a triangle in $\CMS(A)$. It is readily verified that support of  $\sHom(U, U)$ is contained in the intersection of support of $\sHom(M,M)$ and $M/x M$. So $\dim \Supp(U) \leq \dim \Supp(M) -1$. It is also not difficult to prove that $U$ is not a free $A$-module. As $\X(N)$ is thick we get $U \in \X(N)$. By induction hypotheses  $\deg t_{I, N}(U,-) = r_{I, N}^0$. By the structure of triangles in $\CMS(A)$, see \cite[4.4.1]{Buchw}, we have an exact sequence
$0 \rt G \rt U \rt M/xM \rt 0$ with $G$-free. It follows that $L_3(U) = L_3(M/xM)$. We also have an exact sequence
$0 \rt M \xrightarrow{x} M \rt M/xM \rt 0$. As $x \in \ann L_i(M)$ it follows that we have an exact sequence
$$ 0 \rt L_3(M) \rt L_3(M/xM) \rt L_2(M) \rt 0.$$
As the Hilbert function of $L_3(M)$ and $L_2(M)$ are identical, see  \ref{basic}(2), we get that
$2 t_{I, N}(M, -) = t_{I, N}(U,-)$. It follows that $\deg t_{I, N}(M,-) = r_{I, N}^0$. By induction the result follows.

Set $E^i(D) = \bigoplus_{n \geq 1}\Ext_A^i(D, N/I^{n}N)$ for $i \geq 1$. We note that $E^ i(D)$ is a $\R(I)$-module.

Next we show $\deg s_{I,N}(M, -) = s_{I, N}^0$ if $M \in \X(N)$ is non-free.
We prove this assertion by induction on $\dim \Supp(M)$. If $\dim \Supp(M) = 0$ then $M$ is free on $\Spec^0(A)$. In this case we have nothing to show.

Now assume $\dim \Supp(M) > 0$.
By \ref{ann} there exists $l$ such that $\m^l E^i(M) = 0$ for all  for all $i \geq 1$.
Let
$$x \in \m^l \setminus \bigcup_{ \stackrel{P \supseteq \ann \sHom(M, M)}{P \ \text{minimal}}} P.$$
Let $M \xrightarrow{x} M \rt U \rt \Om^{-1}(M)$ be a triangle in $\CMS(A)$. As before we have $\dim \Supp(U) \leq \dim \Supp(M) -1$ and $U$ is not free. As $\X(N)$ is thick we get $U \in \X(N)$. By induction hypotheses  $\deg s_{I,N}(U, -) = s_{I, N}^0$. By the structure of triangles in $\CMS(A)$, see \cite[4.4.1]{Buchw}, we have an exact sequence
$0 \rt G \rt U \rt M/xM \rt 0$ with $G$-free. It follows that $E^3(U) = E^3(M/xM)$. We also have an exact sequence
$0 \rt M \xrightarrow{x} M \rt M/xM \rt 0$. As $x \in \ann E^i(M)$ it follows that we have an exact sequence
$$ 0 \rt E^2(M) \rt E^3(M/xM) \rt E^3(M) \rt 0.$$
As the Hilbert function of $E^3(M)$ and $E^2(M)$ are identical, see \ref{basic}(2), we get that
$2  s_{I,N}(M, -) = s_{I,N}(U, -)$. It follows that $\deg s_{I,N}(M, -) = s_{I, N}^0$. By induction the result follows.
\end{proof}

\section{Proof of Theorem \ref{m2-gen} }
In this section we give a proof of Theorem \ref{m2-gen}. We first need to do some base change. We show that it suffices to assume $A$ is complete with infinite residue field.
\s\label{bc} We need to show that it suffices to assume $A$ is complete with infinite residue field. If the residue field of $A$ is finite then set $B = A[X]_{\m A[X]}$. The maximal ideal of $B$ is $\n = \m B$ and the residue field of $B$ is $l = k(X)$ is infinite. It is clear that $B$ is \CM \ of dimension $d = \dim A$. Also as $\n = \m B$ the maximal ideal of $B$ is generated by $d + 1$ elements.  So $B$ is also a hypersurface ring. Let $(C, \q)$ be the completion of $B$.  We note that we have a flat map $A \rt C$ with $\m C = \q$. We note that if $N$ is a perfect $A$-module then
$N\otimes_A C$ is a perfect $C$-module. If $M$ is a MCM $A$-module then $M\otimes_A C$ is a MCM $C$-module. Also it is clear that if $E$ is an $A$-module of finite length then $E\otimes_A C$ has finite length as a $C$-module and $\ell_A(E) = \ell_C(E\otimes_A C)$. It follows that if $M$ is a MCM $A$-module then
$e_{I,N}(M, -) = e_{IC,N\otimes_A C}(M\otimes_A C, -)$. Let $d = \dim A = \dim C$. If $D = \Omega^d_A(k)$ then $D \otimes_A C = \Omega^d_C(l)$. It follows that $e^0_{I, N} = e^0_{IC, N\otimes_A C}$. Thus we may assume $A$ is complete with infinite residue field.

\s\label{setup} Let $(A,\m)$ be a complete \CM \ local ring with infinite residue field $k$ and dimension $d$. We discuss module structures of the relevant modules. Fix $N$ a perfect $A$-module and an ideal $I \subseteq \m$ with
$\ell(N/IN)$ finite.
  If $M$ is a finitely generated $A$-module then set $$U(M) = \bigoplus_{n \geq 0}\Ext^{d+1}_A( N/I^{n+1}N, M).$$ Set $U(M)_n = \Ext^{d+1}_A(N/I^{n+1}N, M)$. We note that $U$ has a natural structure of a $\R(I)$-module. However if $at \in \R(I)_1$ then $atU(M)_n \subseteq U(M)_{n-1}$. \emph{To make  $U(M)$ as a graded $\R(I)$-module we consider $U(M)_n$ to sit in degree $-n$.} Let $E$ be the injective hull of $k = A/\m$ and if $D$ is an $A$-module set $D^\vee = \Hom_A(D, E)$. We note that if $\ell(D) <  \infty $ then $\ell(D^\vee) = \ell(D)$. Set $V(M)_n = U(M)_n^\vee$. Set
$$V(M) = \bigoplus_{n \geq 0} V(M)_n = \bigoplus_{n \geq 0}\Ext^{d+1}_A( N/I^{n+1}N, M)^\vee.$$
We set $\deg V(M)_n = n$. Note if $at \in \R(I)_1$ then $atV(M)_n \subseteq V(M)_{n+1}$.
So $V(M)$ is a graded $\R(I)$-module.
We first show:
\begin{theorem}\label{dual}(with hypotheses as in \ref{setup}. $V(M)$ is a finitely generated $\R(I)$-module.
\end{theorem}
We will need the following result.
\begin{lemma}
\label{Noeth} Let $W = \bigoplus_{n \geq 0}W_n$ be a graded $\R(I)$-module with $\ell_A(W_n)$ finite for all $n \geq 0$. Let $a \in \R(I)_1$. If $W/aW$ is a finitely generated $\R(I)$-module then
$W$ is a finitely generated $\R(I)$-module.
\end{lemma}
\begin{proof}
Let $\{ \ov{w_1}, \ldots, \ov{w_s}\}$ be a set of homogeneous generators of $W/aW$. Let $r = \max \{ \deg \ov{w_i} \mid 1 \leq i \leq s \}$. Set $Y = $ submodule of $W$ generated by elements of $W$ in degrees $0$ to $r$. Clearly $Y$ is a finitely generated submodule of $W$. It suffices to show $W = Y$. We prove $Y_n = W_n$ by induction on $n \geq 0$. For $0 \leq n \leq r$ we have nothing to prove. Let $n > r$ and assume $Y_j = W_j$ for $j < n$. Let $w \in W_n$. Then $\ov{w} \in W/aW$. Then $\ov{w} = \sum_{i = 1}^{s}\alpha_i\ov{w_i}$.
 Set $y  = \sum_{i = 1}^{s}\alpha_iw_i$. Note $y \in Y$ by our construction. We have $\ov{y} = \ov{w}$. So $y - w = a t$. Note $\deg t < n$. So $t \in Y$ by our induction hypothesis. Thus $w \in Y$. The result follows by induction.
\end{proof}
We now give
\begin{proof}[Proof of Theorem \ref{dual}] Set $V = V(M)$.
We note that $\ell(V_n) < \infty$ for all $n \geq 0$. We prove the result by induction on $r = \dim N$.

 If $r = 0$ then $I^n N = 0$ for $n \gg 0$. So $V_n = \Ext^{d+1}_A(N, M)^\vee = 0$ for $n \gg 0$ as $\projdim N < \infty$. The result follows in this case.

Assume $r \geq 1$ and the result is proved for perfect modules of dimension $r -1$. As $r \geq 1$ and $k$ is infinite there exists $x \in \R(I)_1$ which is $N$-superficial \ \wrt \ $I$. We note that as $\depth N > 0$ we have $(I^{n+1}N \colon x) = I^nN$ for $n \gg 0$; say for $n \geq n_0$.  Set $\ov{N} = N/xN$. Notice $\ov{N}$ is a perfect $A$-module of dimension $r - 1$. We have a natural exact sequence
\begin{align*}
  0 \rt \frac{(I^{n+1}N \colon x)}{I^n N} \rt \frac{N}{I^nN} &\xrightarrow{\alpha_n}  \frac{N}{I^{n+1}N} \rt  \frac{\ov{N}}{I^{n+1}\ov{N}} \rt 0\\
  \alpha_n(w + I^nN) &= xw + I^{n+1}N.
\end{align*}
We note that $\alpha_n$ is injective for $n \geq n_0$. Notice for $n \geq n_0$ we have an exact sequence
\[
\Ext^{d+1}_A(N/I^n N, M)^\vee \rt \Ext^{d+1}_A(N/I^{n+1} N, M)^\vee \rt \Ext^{d+1}_A(\ov{N}/I^{n+1} \ov{N}, M)^\vee.
\]
 Set $\ov{V} = V/xtV$. We note that we have an inclusion of $\R(I)$-modules
 \[
 \ov{V}_{n \geq n_0} \subseteq \bigoplus_{n\geq 0} \Ext^{d+1}_A(\ov{N}/I^{n+1} \ov{N}, M)^\vee.
 \]
 By induction hypothesis the latter module is a  finitely generated $\R(I)$-module. So $\ov{V}_{n \geq n_0}$ is finitely generated $\R(I)$-module. As $\ov{V}_i$ has finite length for all $i$, it follows that $\ov{V}$ is a finitely generated $\R(I)$-module. By Lemma \ref{Noeth}, it follows that $V$ is a finitely generated $\R(I)$-module.
\end{proof}
As a consequence we obtain:
\begin{corollary}\label{ann-e}(with  hypotheses as in \ref{setup}). We have
\begin{enumerate}[\rm (1)]
  \item There exists $s$ such that $\m^s U(M)_n = 0$ for all $n \geq 0$.
  \item The function $n \rt \ell(U(M)_n)$ is of polynomial type of degree $\leq \dim N - 1$.
\end{enumerate}
\end{corollary}
\begin{proof}
  (1) As $V(M)$ is a finitely generated $\R(I)$-module and as $\ell(V(M))_n$ has finite length for all $n \geq 0$ it follows that there exists $s$  such that $\m^s V(M)_n = 0$ for all $n \geq 0$.
  The result follows as $\ann U(M)_n = \ann V(M)_n$ for all $n \geq 0$.

  (2) Let $B = A/\ann(N)$. Then $IB$ is primary to the maximal ideal of $B$. We note that $V(M)$ is a finitely generated $B[It]$-module. It follows that $n \rt \ell(V_n)$ is of polynomial type of degree $\leq \dim B - 1$. As $N$ is a faithful $B$-module we have $\dim B = \dim N$. The result follows as $\ell(U(M)_n) = \ell(V(M)_n)$.
\end{proof}
We now give proof of Theorem \ref{m2-gen}. We restate it for the convenience of the reader.
\begin{theorem}\label{m2-gen-body}
Let $(A,\m)$ be a hypersurface ring of dimension $d$. Let $N$ be a perfect $A$-module and let $I$ be an ideal in $A$ with $\ell(N/IN)$ finite. Then
 there exists integer $e_{I, N}^0$  ($\geq -1$) depending only on $I$ and $N$ such that if $M \in \CMS(A)$ is  non-free then
$\deg e_{I,N}(M, -) = e_{I, N}^0$.
\end{theorem}
\begin{proof}
 If $D$ is a finitely generated $A$-module then set $$U^i(D) = \bigoplus_{n \geq 0}\Ext^{i}_A( N/I^{n+1}N, M).$$
We prove our result by induction on $\dim \Supp(M)$. If $\dim \Supp(M) = 0$ then $M$ is free on $\Spec^0(A)$. In this case we have nothing to show.

Now assume $\dim \Supp(M) > 0$.
We note that for $j \geq 2$,  $$U^{d+j}(M) = U^{d+j-1}(\Omega(M)).$$
By Lemma \ref{ann-e} there exists $l$ such that  $\m^l U^i(M) = 0$ for all $i \geq d+1$.

Let
$$x \in \m^l \setminus \bigcup_{ \stackrel{P \supseteq \ann \sHom(M, M)}{P \ \text{minimal}}} P.$$
Let $M \xrightarrow{x} M \rt W \rt \Om^{-1}(M)$ be a triangle in $\CMS(A)$. It is readily verified that support of  $\sHom(W, W)$ is contained in the intersection of support of $\sHom(M,M)$ and $M/x M$. So $\dim \Supp(W) \leq \dim \Supp(M) -1$. It is also not difficult to prove that $W$ is not free $A$-module.  By induction hypotheses  $\deg e_{I, N}(W,-) = e_{I, N}^0$. By the structure of triangles in $\CMS(A)$, see \cite[4.4.1]{Buchw}, we have an exact sequence
$0 \rt G \rt W \rt M/xM \rt 0$ with $G$-free. It follows that $U^{d+1}(W) = U^{d+1}(M/xM)$. We also have an exact sequence
$0 \rt M \xrightarrow{x} M \rt M/xM \rt 0$. As $x \in \ann U^i(M)$ for $i \geq d +1$ it follows that we have an exact sequence
$$ 0 \rt U^{d+1}(M) \rt U^{d+1}(M/xM) \rt U^{d+2}(M) \rt 0.$$
As the Hilbert function of $U^{d+1}(M)$ and $U^{d+2}(M)$ are identical, see  \ref{basic}(2), we get that
$2 e_{I, N}(M, -) = e_{I, N}(W,-)$. It follows that $\deg e_{I, N}(M,-) = e_{I, N}^0$. By induction the result follows.
\end{proof}
\section{proof of Theorem \ref{m3-gen}}
In this section we give a proof of Theorem \ref{m3-gen}. To prove the result note that we may assume the residue field $k$ of $A$ is infinite (otherwise take the extension $A[t]_{\m A[t]}$ of $A$). We also assume $A$ is complete.

\s Recall $x \in I$ is $M$-superficial with respect to $I$ if there exists $c$ such that $(I^{n+1}M \colon x) \cap I^cM = I^nM$ for all $n \gg 0$. Assume $\ell(M/IM)$ is finite. If $\depth M > 0$ then it follows that $x$ is $M$-regular and $(I^{n+1}M \colon x) = I^nM$ for $n \gg 0$. If $k$ is infinite then $I$-superficial elements with respect to $M$ exists. In fact in this case there exists a non-empty open set $U_M$ in the Zariski topology of $I/\m I$ such that if the image of $x$ is in $U_M$ then $x$ is $I$-superficial with respect to $M$.

\s Let $E = \bigoplus_{n \geq 0}E_n$ be a finitely generated graded module over the Rees algebra $\R= A[It]$. Assume $E_n$ has finite length for all $n$. There exists $xt \in \R_1$ such that $xt$ is $E$-filter regular, i.e.,  $(0 \colon_E xt)_n = 0$ for $n \gg 0$. In fact in this case there exists a non-empty open set $U_E$ in the Zariski topology of $I/\m I$ such that if the image of $x$ is in $U_E$ then $xt$ is $E$-filter regular.

We will need the following result:
\begin{lemma}\label{perf-tensor}
Let $(A,\m)$ be a Gorenstein local ring and let $N$ be a perfect module of dimension $r$. Let $M$ be a MCM $A$-module. Then $M\otimes N$ is \CM \ $A$-module of dimension $r$.
\end{lemma}
\begin{proof}
We note that if $N \neq 0$ then $M\otimes N \neq 0$. We prove the result by induction on $r = \dim N$. When $r = 0$ then we have nothing to prove.
Assume $r \geq 1$. Let $x\in \m$ be $N$-regular. Then $N/xN$ is a perfect $A$-module. By \ref{pdimf} we have $\Tor^A_1(M, N/xN) = 0$. The exact sequence
$0 \rt N \xrightarrow{x} N \rt N/xN \rt 0$ yields an exact sequence
$$0 = \Tor^A_1(M, N/xN) \rt M\otimes N \xrightarrow{x} M\otimes N  \rt M\otimes (N/xN) \rt 0. $$
So $x$ is $M\otimes N$-regular. Also $M\otimes N/xN$ is \CM \ of dimension $r-1$ by induction hypothesis. The result follows.
\end{proof}
\s \label{fin-3thm} Assume $A$ is Gorenstein, $N$ is a perfect $A$-module and $I$ is an ideal of $A$ with $\ell(N/IN)$ finite. Let $M$ be a MCM $A$-module.
For $i \geq 1$
set $L_i^N(M) = \bigoplus_{n \geq 1}\Tor^A_i(M, N/I^{n}N)$ and $E^i_N(M) = \bigoplus_{n \geq 1}\Ext_A^i(M, N/I^{n}N)$.
\begin{proposition}\label{rachel}
(with hypotheses as in \ref{fin-3thm}) For $i \geq 1$, $L_i^N(M)$ and $E^i_N(M)$ are finitely generated $\R(I)$-modules.
\end{proposition}
\begin{proof}
Set $L_i(M) = L_i^N(M)$ and $E_N^i(M) = E^i(M)$.
  Set $W = \bigoplus_{n \geq 1}N/I^{n}N$. Let $\R_I(N) = \bigoplus_{n \geq 0}I^nN$ be the Rees module of $N$ \wrt \ $I$. Clearly $\R_I(N)$ is a finitely generated $\R(I)$-module. We also note that $N[t] = A[t]\otimes N$ is also a $\R(I)$-module (not finitely generated as a $\R(I)$-module). The exact sequence
  $$ 0 \rt \R_I(N) \rt N[t] \rt W \rt 0$$
  yields an $\R(I)$-module structure on  $W$. As $L_i(M) = \Tor^A_i(M, W)$ and $E^i(M) = \Ext_A^i(M, W)$ for $i \geq 1$; it follows that  $L_i(M)$ and $E^i(M)$ are $\R(I)$-modules.

  Tensor the above exact sequence with $M$. As $\Tor^A_i(M, N) = 0 $ for $i \geq 1$, see \ref{pdimf},  we obtain   for $i \geq 2$,  $L_i(M) = \Tor^A_{i-1}(M, \R_I(N))$ and $L_1(M)$ is a $\R(I)$-submodule of $M\otimes \R_I(N)$. It follows that
  $L_i(M)$  are finitely generated $\R(I)$-modules for $i \geq 1$.

  We apply $\Hom_A(M, -)$ to the above exact sequence. As $\Ext^i_A(M, N) = 0$ for $i \geq 1$ we obtain $E^i(M) = \Ext^{i+1}_A(M, \R_I(N))$. It follows that
  $E^i(M)$ are finitely generated $\R(I)$-modules for $i \geq 1$.
\end{proof}
\s\label{mod-TE} Assume $r = \dim N  \geq 2$ and assume $M$ is a MCM $A$-module. Also assume $A$ is Gorenstein and $N$ is perfect $A$-module.  Let $x \in I$ be such that it is $N\oplus (M\otimes N)$-superficial and  $xt$ is $L_1^N(M)$, $E^i_N(M)$-filter regular for $i = 1, 2$. We note that $(I^{n+1}N \colon x) = I^nN$ and $(I^{n+1}(M \otimes N) \colon x) = I^n(M \otimes N)$ for $n \gg 0$ (here we are using
$M\otimes N$ is \CM \ see \ref{perf-tensor}). We note that $\ov{N} = N/xN$ is a perfect $A$-module.
We have an exact sequence for $n \geq 1$
\[
 0 \rt \ker \alpha_n \rt N/I^nN \xrightarrow{\alpha_n} N/I^{n+1}N \rt \ov{N}/I^{n+1}\ov{N} \rt 0,
\]
where $\alpha_n(a + I^n) = xa + I^{n+1}$. We note that $\ker \alpha_n  = (I^{n+1}N \colon x)/I^nN = 0$ for $n \gg 0$.  Thus for $n \gg 0$ we have an exact sequence
\begin{align*}
  \Tor^A_1(M, N/I^nN) &\xrightarrow{\alpha^1_n} \Tor^A_1(M, N/I^{n+1}N) \rt \Tor^A_1(M, \ov{N}/J^{n+1}\ov{N}) \rt \\
 M\otimes N/I^n(M\otimes N) &\xrightarrow{\alpha^0_n}  M\otimes N/I^{n+1}(M\otimes N) \rt \cdots.
\end{align*}
We note that $\ker \alpha^0_n = (I^{n+1}(M\otimes N) \colon x)/I^n(M\otimes N) = 0$ for $n \gg 0$. Furthermore the map $\alpha^1_n$ is the $n^{th}$-component of the multiplication map by $xt \in \R_1$ on $L_1^N(M)$. As $xt$ is $L_1^N(M)$-filter regular it follows that $\ker \alpha^1_n = 0$ for $n \gg 0$. Thus for $n \gg 0$ we have an exact sequence
\s\label{t}
 $$0 \rt  \Tor^A_1(M, N/I^nN) \xrightarrow{\alpha^1_n} \Tor^A_1(M, N/I^{n+1}N) \rt \Tor^A_1(M, \ov{N}/J^{n+1}\ov{N}) \rt 0.$$

Similarly as $xt$ is also $E^1_N(M)$ and $E^2_N(M)$-filter regular we have an exact sequence for $n \gg 0$
\s\label{e}
$$0 \rt  \Ext^1_A(M, N/I^nN) \xrightarrow{\beta^1_n} \Ext_A^1(M, N/I^{n+1}N) \rt \Ext_A^1(M, \ov{N}/I^{n+1}\ov{N}) \rt 0.$$

\s\label{raki}
Assume  $(A,\m)$ is a complete hypersurface ring with an infinite residue field $k$. Let $N$ be a perfect $A$-module and $I \subseteq \m$ is an ideal of $A$ with $\ell(N/IN)$-finite. Let
$M$ be a MCM $A$-module. Set  $$U^i_N(M) = \bigoplus_{n \geq 0}\Ext^{i}_A(N/I^{n+1}N, M) \  \ \text{for $i \geq d +1$}.$$ Also set $(U^i_N(M))_n  = \Ext^{i}_A(N/I^{n+1}N, M)$. Let $E = E(k)$
 Let $E$ be the injective hull of $k = A/\m$ and if $D$ is an $A$-module set $D^\vee = \Hom_A(D, E)$. We note that if $\ell(D) <  \infty $ then  $\ell(D^\vee) = \ell(D)$. Set $(V^i_N(M))_n = (U(M)^i_N)_n^\vee$. Set for $i \geq d +1$
$$V^i_N(M) = \bigoplus_{n \geq 0}(V^i_N(M))_n = \bigoplus_{n \geq 0}\Ext^{i}_A( N/I^{n+1}N, M)^\vee.$$
By \ref{dual} we get that $V^i_N(M)$ is a finitely generated $\R(I)$-module.

\s\label{usha} (with hypotheses as in \ref{raki}). Assume $\dim N \geq 2$. Let $x$ be $N$-superficial \wrt \ $I$. Also assume that $xt$ is $V^{d+1}_N(M) \oplus V^{d+2}_N(M)$-filter regular.
 We note that $(I^{n+1}N \colon x) = I^nN$ for $n \gg 0$. We also note that $\ov{N} = N/xN$ is a perfect $A$-module.
We have an exact sequence for $n \geq 1$
\[
 0 \rt \ker \alpha_n \rt N/I^nN \xrightarrow{\alpha_n} N/I^{n+1}N \rt \ov{N}/I^{n+1}\ov{N} \rt 0,
\]
where $\alpha_n(a + I^n) = xa + I^{n+1}$. We note that $\ker \alpha_n  = (I^{n+1}N \colon x)/I^nN = 0$ for $n \gg 0$.  Thus for $n \gg 0$ we have an exact sequence
for $i \geq d+1$
\[
\cdots \rt U^i_{\ov{N}}(M)_n \rt U^i_N(M)_{n} \xrightarrow{\beta^{i}_n} U^i_N(M)_{n-1} \rt \cdots
\]
Here $\beta^{i}_n$ is multiplication by $xt \in \R(I)_1$. As $xt $ is $V^i_N(M)$-filter regular for $i \geq d +1$ (here we are using the fact that $M$ is two periodic) we get that $\beta^{i}_n$ is surjective for $n \gg 0$. In particular we get an exact sequence for $i \geq d +2$
\[
0 \rt U^i_{\ov{N}}(M)_n \rt U^i_N(M)_{n} \xrightarrow{\beta^{i}_n} U^i_N(M)_{n-1} \rt 0.
\]
As $M$ is $2$-periodic the above exact sequence also holds for $i = d+1$.

We will need the following result:
\begin{proposition}
\label{hwm} Let $(A,\m)$ be a hypersurface of dimension $d$ and let $M$ be a non-free MCM $A$-module. Let $W$ be a finite length $A$-module. The following assertions are equivalent:
\begin{enumerate}[\rm (1)]
  \item $\projdim W < \infty$.
  \item $\Tor^A_i(M, W) = 0$ for $i > 0$.
  \item $\Ext_A^i(M, W) = 0$ for $i > 0$.
  \item $\Ext_A^i(W, M) = 0$ for $i > d$.
\end{enumerate}
\end{proposition}
\begin{proof}
By \ref{pdimf} we have (1) $\implies$ (2), (3). Also the assertion (1) $\implies$ (4) is trivial.

 (2) $\implies$ (1) This follows from \cite[1.9]{HW} (also see \cite[1.1]{M}).

 (3) $\implies$ (1) This follows from \cite[5.12]{AB}.

 (4) $\implies$ (1) This also  follows from \cite[5.12]{AB}.
\end{proof}

We now give a proof of Theorem \ref{m3-gen}. We restate it here for the convenience of the reader.
\begin{theorem}\label{m3-gen-body}
Let $(A,\m)$ be a hypersurface ring of dimension $d$. Let $N$ be a perfect $A$-module and let $I$ be an ideal in $A$ with $\ell(N/IN)$ finite.  Then
 $$ r_{I, N}^0 = s_{I, N}^0 = e_{I, N}^0. $$
\end{theorem}
\begin{proof}
We may assume that $A$ is complete with infinite residue field.
We prove the result by induction on dimension $r$ of $N$.

 Let $M$ be a MCM $A$-module with no free summands.  By Corollary \ref{m1-gen} \\
$\deg t_{I,N}(M, -) = r_{I, N}^0$ and $\deg s_{I,N}(M, -) = s_{I, N}^0$.  By \ref{m2-gen} we have \\$\deg e_{I,N}(M, -) = e_{I, N}^0$.
Recall the degree of the zero polynomial is set as $-1$.

 We first consider the case when $r = 1$.  We note that by \ref{rachel}; $r_{I, N}^0 = 0$ or $ = -1$ and   $s_{I, N}^0 = 0 $ or $ = -1$. We also have $e_{I, N}^0 = 0$ or $=-1$; see \ref{ann-e}.

 We show that  $r_{I, N}^0 = -1$ if and only if $s_{I, N}^0  = -1$  if and only if $e_{I, N}^0 = -1$.

 Suppose $r_{I, N}^0 = -1$.
Then $t_{I,N}(M, -) = 0$ and therefore
 $t_{I, N}(M, n) = 0$ for $n \gg 0$ (say from $n \geq m$). As $t_{I, N}(M, n) = t_{I, N}(\Omega^1_A(M), n)$ for all $n$ and $M$ is $2$-periodic we have $\Tor^A_i(M, A/I^{n +1}) = 0$ for all $i \geq 1$ and for $n \geq m$. It follows from \ref{hwm} that $\projdim_A N/I^{n+1}N < \infty$ for all $n \geq m$. Therefore   $s_{I, N}(M, n) = 0$ and $e_{I,N}(M, -) = 0$ for all $n \geq m$.
So $\deg s_{I,N}(M, -) = -1$ and $\deg e_{I,N}(M, -) = -1$.  So $s_{I, N}^0  = -1$  and $e_{I, N}^0 = -1$.

 The proof of that if $s_{I, N}^0  = -1$ ($e_{I, N}^0 = -1$)  then $e_{I, N}^0 = -1$ and $r_{I, N}^0 = -1$ ($s_{I, N}^0  = -1$ and $r_{I, N}^0 = -1$) respectively is similar and left to the reader.

  We now assume that $r \geq 2$ and the result is proved when $\dim N = r -1$.
A proof similar to above shows that
 $r_{I, N}^0 = -1$ if and only if $s_{I, N}^0  = -1$  if and only if $e_{I, N}^0 = -1$.

So assume  $\deg  t_{I,N}(M, -) \geq 0$ (and so  $\deg s_{I,N}(M, -)  \geq 0 $ and $\deg e_{I,N}(M, -) \geq 0$). Let $x \in I$ be $N \oplus M\otimes N$-superficial and also assume that $xt \in \R_1$ is $L_1^N(M) \oplus E^1_N(M) \oplus E_N^2(M) \oplus V^{d+1}_N(M) \oplus V^{d+1}_N(M)$-filter regular.
We note that   $\ov{N} = N/xN$ is a perfect $A$-module of dimension $r -1$. So it follows from our induction hypotheses that
$$\deg t_{I,\ov{N}}(M, -) = \deg s_{I,\ov{N}}(M, -) = \deg e_{I,\ov{N}}(M, -).$$

As $\deg t_{I,N}(M, -) \geq 0$, it follows from  \ref{t} that $$\deg t_{I,N}(M, -) = \deg t_{I,\ov{N}}(M, -)  +  1.$$ Similarly  as $\deg s_{I,N}(M, -) \geq 0$ it follows from \ref{e} that
$$\deg s_{I,N}(M, -) = \deg s_{I,\ov{N}}(M, -) + 1.$$  Also as $\deg e_{I,N}(M, -) \geq 0$ it follows from  \ref{usha} that
$$\deg e_{I,N}(M, -) = \deg e_{I,\ov{N}}(M, -) + 1.$$
  The result follows.
  \end{proof}


\end{document}